\documentclass[11pt,a4paper]{article}
\usepackage[affil-it]{authblk}
\usepackage{amsthm,amsmath,amssymb}
\usepackage{geometry}

\newcommand\later[1]\empty

\newcommand\val[1]{\left\lVert#1\right\rVert}
\newcommand\hide[1]{\textbf{#1}}
\renewcommand\hide[1]\empty

\newcommand\todo[1]{ [\textbf{!}#1]}
\renewcommand\todo[1]\empty

\newcommand\cone[2]{{#1}\!\left<#2\right>}

\def\mM{\modelsts{M}}
\newcommand\modelsts[1]{\framets{#1}}
\def\tiff{\textrm{ iff }}

\def\B{\mathbf{B}}
\def\Bik{\B_{i,k}}
\newcommand\frametop[2]{{#1{[#2]}}}
\def\frFh{\frametop{\frF}{h}}

\def\Wh{\overline{W}}
\def\Rh{\overline{R}}

\def\IL{\logicnamets{IL}}
\def\CL{\logicnamets{CL}}

\def\Int{{\textrm{(i)}}}
\def\Int{{\textsc{i}}}

\newcommand\axiomts[1]{{\mathrm{#1}}}
\def\Di{\Diamond}
\def\vf{\varphi}
\def\lra{\leftrightarrow}

\newcommand\framets[1]{\mathrm{{#1}}}
\def\frF{\framets{F}}
\def\frA{\framets{A}}
\def\frG{\framets{G}}

\def\AA{\forall}
\def\Imp{\Rightarrow}

\def\h{\operatorname{ht}}

\def\mo{\vDash}
\def\imp{\rightarrow}
\def\vd{\vdash}

\def\emp{\emptyset}

\def\vL{\axiomts{L}}

\def\restr{{\upharpoonright}}

\newcommand\logicts[1]{{\textsc{#1}}}
\newcommand\logicnamets[1]{{\logicts{#1}}}
\newcommand{\LK}[1]{\logicnamets{K#1}}
\newcommand{\LS}[1]{\logicnamets{S#1}}
\newcommand{\Tra}[1]{\logicnamets{Tr}_{#1}}
\def\wK4{\logicnamets{wK4}}
\def\DL{\logicnamets{DL}}

\def\WS5{\logicnamets{WS5}}
\def\MIPC{\logicnamets{MIPC}}

\def\iff{\;\Leftrightarrow\;}

\theoremstyle{definition}

\newtheorem{theorem}{Theorem}

\newtheorem{proposition}{Proposition}

\newtheorem*{remark*}{Remark}

\newtheorem*{question*}{Question}
\newtheorem{definition}{Definition}
\newtheorem{example}{Example}

\newtheorem*{theorem*}{Theorem}

\title{Glivenko's theorem, finite height, and local tabularity}

\author{Ilya B. Shapirovsky}
\affil{\small
Steklov Mathematical Institute of Russian Academy of Sciences \\
and\thanks{The work on this paper was supported by the Russian Science Foundation under grant 16-11-10252 and carried out at Steklov Mathematical Institute of Russian Academy of Sciences.}   \\
Institute for Information Transmission Problems of Russian Academy of Sciences
 }

\begin{document}

\maketitle

\begin{abstract}
Glivenko's theorem states that
a formula is derivable in classical propositional logic $\CL$ iff under the double negation it is derivable in
intuitionistic propositional logic $\IL$:
$\CL\vd\vf$ iff $\IL\vd\neg\neg\vf$.
Its analog for the modal logics $\LS{5}$ and $\LS{4}$ states that
$\LS{5}\vd \vf$   iff  $\LS{4} \vd \neg \Box \neg \Box \vf$.
In Kripke semantics, $\IL$ is the logic of partial orders, and  $\CL$  is the logic of partial orders of height 1. Likewise,
$\LS{4}$ is the logic of preorders, and $\LS{5}$ is the logic of equivalence relations, which are preorders of height 1.
In this paper we generalize Glivenko's translation for logics of arbitrary finite height.
\end{abstract}


\section{Introduction}
For a modal or intermediate logic $\vL$, let
$\vL[h]$ be its extension with the formula restricting the height of a Kripke frame by finite $h$.
In the intermediate case,  such formulas are defined as $B^\Int_0=\bot$,  $B^\Int_{h} = {p_{h} \vee (p_{h} \imp B^\Int_{h-1})}$,
and in the modal transitive case as $B_0=\bot$,  $B_{h} = p_{h} \to  \Box(\Di p_{h} \lor B_{h-1})$. In particular,
classical logic $\CL$ is the extension of intuitionistic logic $\IL$ with the formula $p\vee \neg p$, that is
$\CL=\IL[1]$. Similarly, $\LS{5}=\LS{4}[1]$. 
Glivenko's translation \cite{GLI29a} and its analog for the modal logics $\LS{5}$ and  $\LS{4}$ \cite{Mats-s4s5}
can be formulated as follows:
\begin{eqnarray}\label{eq:Gliv}
\IL[1] \vd \vf &\tiff &\IL \vd \neg\neg \vf, \\
\label{eq:Mats}
\LS{4}[1] \vd \vf &\tiff & \LS{4} \vd \Di \Box \vf.
\end{eqnarray}
For finite variable fragments of $\IL$ and $\LS{4}$, the above equivalences can be generalized for arbitrary finite height.
A \mbox{\em $k$-formula} is a formula in variables $p_0,\ldots p_{k-1}$.
Let $(W,R)$ be the $k$-generated canonical frame of $\LS{4}$ (that is, $W$ is the set of maximal $\LS{4}$-consistent sets of $k$-formulas).
It follows from \cite{ShehtRigNish}
(see also  \cite{ShehtPhD}, \cite{Fine85}, \cite{Bellissima1985}, \cite{Bellissima1986Ht})
that  there exist formulas $\B_{h,k}$ (and their intuitionistic analogs $\B_{h,k}^\Int$)
such that for every $x\in W$,
$\B_{h,k}\in x$  iff  the depth of $x$ in  $W$  is less than or equal to $h$.
We observe that for all finite $k$, for all $k$-formulas $\vf$,
\begin{eqnarray}
\IL[h+1]\vd\vf  &\tiff  &\IL\vd \bigwedge_{i\leq h} ((\vf\imp \Bik^\Int)\imp \Bik^\Int),\\
\label{eq:S4h}
\LS{4}[h+1] \vd\vf  &\tiff &\LS{4} \vd \bigwedge_{i\leq h} (\Box(\Box  \vf\imp \Bik)\imp \Bik).
\end{eqnarray}
In particular, for $h=0$ we have equivalences (\ref{eq:Gliv}) and (\ref{eq:Mats}), since
the formulas $\B_{0,k}$ and $\B_{0,k}^\Int$  are $\bot$ for all $k$.

\smallskip
Sometimes, analogs of the formulas $\B_{h,k}$ exist
for unimodal logics smaller than $\LS{4}$ and for polymodal logics.
A modal logic $\vL$ is  pretransitive (or weakly transitive, in another terminology), if the  transitive reflexive closure modality $\Di^*$ is expressible in $\vL$ \cite{KrachtBook}.
Namely, for the language with  $n$ modalities $\Di_i$ ($i<n$),  put $\Di^0 \vf = \vf$, $\Di^{m+1}\vf=\Di^m \vee_{i< n} \Di_i\vf$, $\Di^{\leq m} \vf =
\vee_{l\leq m} \Di^l \vf$.
A logic $\vL$  is {\em pretransitive} if
it contains $\Di^{m+1} p \imp \Di^{\leq m} p$ for some finite $m$.
In this case $\Di^{\leq m}$ plays the role of $\Di^*$.
The {\em height of a polymodal frame $(W,(R_i)_{i<n})$} is the height of the preorder
$(W, (\bigcup_{i<n}R_i)^*)$.
In the pretransitive case,
formulas of finite height can be defined analogously to the transitive case.

$\vL$ is said to be {\em $k$-tabular} if, up to the equivalence in $\vL$, there exist
only finitely many $k$-formulas. $\vL$ is {\em locally tabular} (or {\em locally finite}) if it is $k$-tabular for every finite $k$.

We show (Theorems \ref{thm:topheavy} and \ref{thm:main}) that if
$\vL$ is a pretransitive logic, $h,k< \omega$, and $\vL[h]$ is $k$-tabular, then:
\begin{enumerate}
\item
For every $i\leq h$,  there exists a formula $\Bik$ such that
$\Bik\in x$ iff   the depth of $x$ in the $k$-generated canonical frame of $\vL$  is less than or equal to $i$.
\item  For all $k$-formulas $\vf$,
\begin{equation}\label{eq:mainTransl}
\vL[h+1] \vd\vf \tiff \vL \vd \bigwedge_{i\leq h} (\Box^*(\Box^* \vf\imp \Bik)\imp \Bik).
\end{equation}
\end{enumerate}

The equivalence (\ref{eq:mainTransl}) 
generalizes (\ref{eq:S4h}). Recall that
a unimodal transitive logic is locally tabular iff it is of finite height iff it is 1-tabular
(\cite{Seg_Essay}, \cite{Max1975}).
In the non-transitive case the situation is much more complicated.
It follows from \cite{LocalTab16AiML} that every locally tabular (even 1-tabular) logic is a pretransitive logic of finite height;
however, it follows from
\cite{Makinson81} that
there exists a pretransitive $\vL$  such that
none of the logics $\vL[h]$  are 1-tabular.
In Section \ref{sec:coroll}  we discuss how $k$-tabularity of $\vL[h]$ depends on $h$ and $k$.
In particular, we construct the first example of a modal logic which is 1-tabular but not locally tabular.

\section{Preliminaries}
Fix a finite $n>0$; 
{\em $n$-modal formulas} are built from a countable set $\{p_0, p_1, \ldots \}$ of proposition letters,
the classical connectives $\imp$, $\bot$, and the modal connectives $\Di_i$, $i<n$;
the other Boolean connectives are defined as standard abbreviations;
$\Box_i$ abbreviates $\neg\Di_i\neg$. We omit the subscripts on the modalities when $n=1$.
By a {\em logic} we mean a {\em propositional $n$-modal normal logic}, that is a set of $n$-modal formulas
containing all classical tautologies, the formulas \mbox{$\Di_i (p\vee q) \imp \Di_i p\vee \Di_i q$}
and $\neg\Di_i \bot$ for all $i<n$, and closed under the rules of Modus Ponens, Substitution, and Monotonicity
(if $\vf\imp \psi$ is in the logic, then so is $\Di_i \vf\imp \Di_i \psi$).

For a logic $\vL$ and a set of formulas $\Psi$, the smallest logic containing $\vL\cup\Psi$ is denoted by $\vL+\Psi$.
For a formula $\vf$, the notation
$\vL+\vf$ abbreviates $\vL+\{\vf\}$.
In particular,
$\LK{4}=\LK+\Di\Di p\imp \Di p$, $\LS{4}=\LK{4}+{p\imp \Di p}$, $\LS{5}=\LS{4}+p\imp \Di\Box p$, where
$\LK{}$ denotes the smallest unimodal logic.  $\vL\vd \vf$ is a synonym for $\vf\in\vL$.

The truth and the validity of modal
formulas in Kripke frames and models are defined as usual, see, e.g., \cite{blackburn_modal_2002}.
By a {\em frame} we always mean a Kripke frame $(W,{(R_i)}_{i<n})$, $W\neq \emp$, $R_i\subseteq W\times W$.
We put $R_\frF=\cup_{i<n} R_i$.
The transitive reflexive closure of a relation $R$ is denoted by $R^*$; the notation $R(x)$ is used for the set $\{y\mid x R y\}$.
The {\em restriction} of $\frF$ onto its subset $V$, $\frF\restr V$ in symbols, is $(V, (R_i\cap (V\times V))_{i<n})$.
In particular, we put $\cone{\frF}{x} = \frF\restr R^*_\frF(x)$.

For $k\leq \omega$, a {\em $k$-formula} is a formula in proposition letters $p_i$, $i<k$.

Let $\vL$ be a consistent logic. For $k\leq \omega$, the {\em $k$-canonical model of} $\vL$ is built from
maximal $\vL$-consistent sets of $k$-formulas; the relations and the valuation are defined in the standard way, see e.g. \cite{Ch:Za:ML:1997}.
Recall the following fact.
\begin{proposition}[Canonical model theorem]
Let $\mM$ be the  $k$-canonical model of a logic $\vL$, $k\leq \omega$. Then
for all $k$-formulas $\vf$ we have:
\begin{enumerate}
\item
$\mM,x\mo \vf \tiff \vf \in x$, for all $x$ in $\mM$;
\item
$\mM\mo \vf \tiff \vL\vd\vf$.
\end{enumerate}
\end{proposition}


A logic $\vL$ is said to be {\em $k$-tabular} if, up to the equivalence in $\vL$, there exist
only finitely many \mbox{$k$-formulas}. $\vL$ is {\em locally tabular} (or {\em locally finite}) if it is $k$-tabular for every finite $k$.
The following proposition is straightforward from the definitions.
\begin{proposition}\label{prop:LTbasic}
Let $\vL$ be a logic, $k<\omega$.
The following are equivalent:
\begin{itemize}
\item
$\vL$ is  $k$-tabular.
\item
The $k$-generated Lindenbaum-Tarski algebra of $\vL$ is finite.
\item
The $k$-canonical model of $\vL$ is finite.
\end{itemize}
\end{proposition}


\smallskip

A unimodal logic is {\em transitive} if it contains the formula $\Di\Di p\imp \Di p$; recall that this formula expresses
transitivity of a binary relation. Below we consider a weaker property,
{\em pretransitivity} of (polymodal) logics and frames.

For a binary relation $R$ on a set $W$, put $R^{\leq m}=  \cup_{i \leq m} R^i$, where $R^0=Id(W)$, $R^{i+1}=R\circ R^i$.
$R$ is called {\em $m$-transitive}  if $R^{\leq m}=R^*$.
$R$ is {\em pretransitive} if it is $m$-transitive for some $m$.
A frame $\frF$
is {\em $m$-transitive} if $R_\frF$ is $m$-transitive.

Let
$
\Di^0 \vf = \vf, ~\Di^{i+1}\vf=\Di^i (\Di_0\vf \vee \ldots \vee \Di_{n-1}\vf) , ~\Di^{\leq m} \vf =
\vee_{i\leq m} \Di^i \vf,~  \Box^{\leq m} \vf =\neg \Di^{\leq m} \neg \vf$.

\begin{proposition}
Let $\frF$ be a frame. The following are equivalent:
\begin{itemize}
\item 
$\frF$ is $m$-transitive;
\item
$R_\frF^{m+1}\subseteq R_\frF^{\leq m}$;
\item
$\frF\mo \Di^{m+1} p \imp \Di^{\leq m} p$.
\end{itemize}
\end{proposition}
The proof is straightforward, details can be found, e.g., in \cite{KrachtBook}.

\begin{definition}
A logic $\vL$ is said to be {\em $m$-transitive} if
$\vL\vd \Di^{m+1} p \imp \Di^{\leq m} p$.
$\vL$ is {\em pretransitive} if it is $m$-transitive for some $m\geq 0$.\footnote{Pretransitive logics sometimes are called {\em weakly transitive}. However, in the other terminology, the term `weakly transitive' is used for logics containing the formula
$\Di\Di p\imp \Di p\vee p$.}.
\end{definition}

If $\vL$ is pretransitive, then there exists the least $m$ such that
 $\vL$  is $m$-transitive; in this case we write $\Di^* \vf$  for $\Di^{\leq m}\vf$,  and  $\Box^* \vf$  for  $\Box^{\leq m}\vf$ .

 For a unimodal formula  $\vf$,
 $\vf^{[*]}$ denotes the formula obtained from $\vf$ by replacing $\Di$ with $\Di^{*}$ and $\Box$ with $\Box^{*}$.

\begin{proposition}\label{prop:fragmS4}
For a pretransitive logic $\vL$, the set $\{\vf\mid \vL\vd \vf^{[*]} \}$ is a logic containing $\LS{4}$.
\end{proposition}
\begin{proof}
Follows from \cite[Lemma 1.3.45]{ShefSkvGab}.
\end{proof}

\smallskip

A poset is of {\em height} $h<\omega$
if it contains a chain of $h$ elements and no chains of cardinality $>h$.

A {\em cluster} in a frame $\frF$ is an equivalence class with respect to the relation $\sim_\frF\; =\;R^*_\frF\cap {R^*_\frF}^{-1}$.
For clusters $C,~D$, put $C\leq_\frF D$ iff $x R^*_\frF y$ for some $x\in C, y\in D$.
The poset $(W{/}{\sim_\frF},\leq_\frF)$ is called the {\em skeleton of} $\frF$.
The {\em height of a frame} $\frF$, in symbols $\h(\frF)$, is the height of its skeleton.

Put
$$
B_0 =\bot,\quad
B_{i+1} = p_{i+1} \to  \Box^{*}(\Di^{*} p_{i+1} \lor B_i).
$$
In the unimodal transitive case, the formula $B_h$ expresses the fact that the height of a frame $\leq h$ \cite{Seg_Essay}.
In the case when $\frF=(W,(R_i)_{i<n})$ is $m$-transitive,  the operator $\Di^*=\Di^{\leq m}$  relates to $R^*_\frF$.
Since
the height of $\frF$  is the height of the preorder $(W,R^*_\frF)$, we have $\frF\mo B_h$  iff  $\h(\frF)\leq h.$

\begin{definition}
A pretransitive logic is of {\em finite height} if it contains $B_h$ for some $h<\omega$.
For a pretransitive $\vL$,  we put $$\vL[h]=\vL+B_h.$$
\end{definition}

\begin{example}
Unimodal examples of 1-transitive logics are
$\LS{4}$, $\wK4=\LK{}+\Di\Di p\imp \Di p \vee p$.
The logic $\LS{5}$ and the {\em difference logic} $\DL=\wK4+p\imp \Box\Di p$ are examples of logics of height 1.

A well-known logic $\LK{5}=\LK{}+\Di p\imp \Box\Di p$ is a  2-transitive logic of height 2.
To show this, recall  that $\LK{5}$ is Kripke complete and
its frames are those that validate the sentence $\AA x\AA y\AA z (xRy \wedge xRz \rightarrow yRz)$.
Every $\LK{5}$-frame is 2-transitive. Indeed, suppose  that $aRbRcRd$ for some elements of a $\LK{5}$-frame.
Then $bRb$; we also have $bRc$, thus $cRb$; from $cRb$ and $cRd$ we infer that $bRd$. Thus $a R^2 d$.
It is not difficult to see that if a $\LK{5}$-frame $\frF$ has an irreflexive serial point, then the height of $\frF$ is 2; otherwise $\frF$ is a disjoint sum of $\LS{5}$-frames and irreflexive singletons,
so its height is 1.
\end{example}

\begin{theorem}[\cite{Seg_Essay,Max1975}]\label{thm:seg-max-crit}
A unimodal transitive logic is locally tabular iff it is of finite height.
\end{theorem}

In \cite{LocalTab16AiML}, it was shown that every locally tabular unimodal logic is a pretransitive logic of finite height;
in fact, the proof yields the following stronger formulation.
\begin{theorem}
If a logic is 1-tabular, then it is  a pretransitive logic of finite height.
\end{theorem}
\begin{proof}
Let $\vL$ be  1-tabular. Then  its 1-canonical frame is finite. Every finite frame is $m$-transitive for some $m$.
Thus $\vL$ is $m$-transitive.

By Proposition \ref{prop:fragmS4},  the set $^*\vL=\{\vf\mid \vL\vd \vf^{[*]} \}$ is a logic containing $\LS{4}$.
Since $\vL$ is 1-tabular, $^*\vL$ is 1-tabular. In \cite{Max1975}, it was shown that for transitive logics 1-tabularity implies
local tabularity. 
Thus $^*\vL$ is of finite height. It follows that $\vL$ is of finite height too.
\end{proof}

Thus, all locally tabular logics are pretransitive of finite height. However,  unlike the transitive case, the converse is not true in general
even for unimodal logics.
Let $\Tra{m}$ be the smallest $m$-transitive unimodal logic.
For $m\geq 2$, $h\geq 1$, none of the logics $\Tra{m}[h]$  are locally tabular \cite{Byrd78};
moreover, they are not $1$-tabular \cite{Makinson81}.

\section{Translation for logics of height 1}\label{sec:Gliv1}
For a pretransitive logics $\vL$,   $\vL[1] =\vL+B_1$, that is $\vL[1]$ is the smallest logic containing $\vL\cup\{p\imp \Box^*\Di^*p\}$.
It is known that $\LS{4}[1]=\LS{5}\vd \vf$  iff $\LS{4} \vd \Di \Box \vf$, and
$\LS{5}\vd \Box \psi \imp \Box \vf \tiff \LS{4}  \vd \Di \Box \psi \imp \Di \Box \vf$
\cite{Mats-s4s5}, \cite{Rybakov1992}.
In \cite{KudShap2017Engl}, it was shown that in the pretransitive unimodal case we have
$\vL[1]\vd \vf$  iff $\vL \vd \Di^*\Box^*\vf.$
In this section we generalize these facts to the polymodal case using the maximality property of pretransitive canonical frames
(see Proposition \ref{prop:max} below).

\begin{proposition}\label{prop:canon-Di-R}
Let $\frF$ be the $k$-canonical frame of a pretransitive logic $\vL$, $k\leq \omega$. For all  $x,y$ in $\frF$, we have
$$x R_\frF^* y  \;\tiff\;   \AA\vf\;(\vf\in y \;\Imp\; \Di^* \vf\in x).$$
\end{proposition}
The proof is straightforward; for details see, e.g., Proposition 5.9 and Theorem 5.16 in \cite{Ch:Za:ML:1997}.

Consider a frame $\frF$ and its subset $V$. We say that $x\in V$ is a {\em maximal element} of $V$,
if for all $y\in V$,  $xR^*_\frF y$ implies $y R^*_\frF x$.

It is known that in canonical transitive frames every non-empty definable subset has a maximal element \cite{Fine85};
the next proposition shows that this property holds in the pretransitive case as well.
\begin{proposition}[Maximality lemma]\label{prop:max}
Suppose that $\frF$ is the  $k$-canonical frame of a pretransitive $\vL$, $k\leq \omega$.
Let $\vf\in x$ for some $x$ in $\frF$ and  some formula $\vf$. Then $R^*_\frF(x)\cap\{y\mid \vf\in y\}$
has a maximal element.
\end{proposition}
\begin{proof}
For a formula $\alpha$, put $\val{\alpha}=\{y\mid \alpha\in y\}$.
Since $\vf\in x$, $R^*_\frF(x)\cap \val{\vf}$ is non-empty.

Let $\Sigma$ be an $R^*_\frF$-chain in $R^*_\frF(x)\cap \val{\vf}$.
The family
$\{R^*_\frF(y) \cap \val{\vf}\mid  y\in\Sigma \}$
has the finite intersection property
  (indeed, if $\Sigma_0$ is a non-empty finite subset of $\Sigma$, then for some $y_0\in \Sigma_0$ we have $yR^*_\frF y_0$ for all $y\in \Sigma_0$;
  so $y_0\in R^*_\frF(y) \cap \val{\vf}$ for all $y\in\Sigma_0$).
By Proposition \ref{prop:canon-Di-R}, $R_\frF^*(y)=\bigcap\{\val{\alpha}\mid \Box^*\alpha\in y\}$. It follows that all sets $R^*_\frF(y) \cap \val{\vf}$ are closed in the
Stone topology on $\frF$ (see, e.g., \cite[Theorem 1.9.4]{gold:math-mod93}). By the compactness,
$\bigcap \{R^*_\frF(y) \cap \val{\vf}\mid y\in\Sigma\}$ is non-empty. Thus $\Sigma$ has an upper bound in $\val{\vf}$.
By Zorn's lemma, $R^*_\frF(x)\cap \val{\vf}$  contains a maximal element.
\end{proof}

\begin{proposition}\label{prop:cons}
A pretransitive logics $\vL$ is consistent iff $\vL[1]$  is consistent.
\end{proposition}
\begin{proof}
Easily follows from Proposition \ref{prop:fragmS4} and the fact that if
a logic containing $\LS{4}$ is consistent, then its extension with the formula $p\imp \Box\Di p$ is consistent.
\end{proof}
Since $\vL[1]\supseteq\vL[2]\supseteq\vL[3]\supseteq\ldots $, it follows that if $\vL$ is consistent,  then $\vL[h]$
is consistent for any $h>0$.

\smallskip

For a frame $\frF$ and a point $x$ is $\frF$,
the {\em depth of $x$ in $\frF$} is the height of the frame $\cone{\frF}{x}$. Let  $\frFh$ denote the restriction of $\frF$ onto the set of its points
of depth less than or equal to $h$, i.e., $\frFh\;=\;{\frF\restr \{x  \mid \h(\cone{\frF}{x})\leq h\}}$.

\begin{proposition}\label{prop:canon-pointwise}
Let $\frF$ be the $k$-canonical frame of a pretransitive logic $\vL$, $k\leq \omega$.
\begin{enumerate}
\item
For all $x$ in $\frF$, $0\leq h<\omega$,
\[
\text{ the depth of $x$ in $\frF$ is $\leq h$}
\;\tiff\;
B_h(\psi_1,\ldots,\psi_h)\in x \textrm{ for all $k$-formulas } \psi_1,\ldots,\psi_h.
\]
\item
For  $0<h< \omega$, the frame $\frFh$ is the canonical frame of $\vL[h]$.
\end{enumerate}
\end{proposition}
\begin{proof}
1. If $\h(\cone{\frF}{x})\leq h$, then $B_h$ is valid at $x$ in $\frF$; by the Canonical model theorem, $B_h(\psi_1,\ldots,\psi_h)\in x$ for all $k$-formulas $\psi_1,\ldots,\psi_h$.

By induction on $h$, let us show that if $\h(\cone{\frF}{x})>h$, then $B_h(\psi_1,\ldots,\psi_h)\notin x$ for some  $\psi_1,\ldots,\psi_h$.
The basis is trivial, since there are no points containing $B_0=\bot$ in $\frF$.
Suppose $\h(\cone{\frF}{x})>h+1$. Then there exists $y$ such that $\h(\cone{\frF}{y})>h$,
$(x,y)\in R_\frF^*$, and $(y,x)\notin R_\frF^*$.  By induction
hypothesis, $B_h(\psi_1,\ldots,\psi_h)\notin y$ for some  $\psi_1,\ldots,\psi_h$. By Proposition \ref{prop:canon-Di-R}, for some $\psi_{h+1}$ we have
$\psi_{h+1}\in x$ and $\Di^*\psi_{h+1}\notin y$. It follows that $B_{h+1}(\psi_1,\ldots,\psi_{h+1})\notin x$.

\medskip 2. Since $\vL\subseteq \vL[h]$, the $k$-canonical frame of $\vL[h]$ is a generated subframe of $\frF$.
Now the statement follows from the first statement of the proposition.
\end{proof}

A logic $\vL$ is {\em $k$-canonical} if it is valid in its $k$-canonical frame.
\begin{proposition}
If a pretransitive $\vL$ is $k$-canonical ($k\leq \omega$), then  $\vL[h]$ is $k$-canonical for all $0<h<\omega$.
\end{proposition}
\begin{proof}
Follows from Proposition \ref{prop:canon-pointwise}.
\end{proof}

\begin{theorem}\label{thm:embedd}
Let  $\vL$ be a pretransitive logic. Then for all formulas $\vf,\psi$ we have
$$
\vL[1] \vd \Box^* \psi\imp \Box^* \vf \tiff \vL  \vd \Di^*\Box^*\psi\imp \Di^*\Box^* \vf.$$
\end{theorem}
\begin{proof} By Proposition \ref{prop:cons}, we may assume that both $\vL$ and $\vL[1]$ are consistent.
Let  $\frF$ be the $\omega$-canonical frame of $\vL$, and $\frG$  the $\omega$-canonical frame of $\vL[1]$.

Suppose $\vL  \vd \Di^*\Box^* \psi\imp \Di^*\Box^* \vf$.
Consider an element $x$ of $\frG$.
By Proposition \ref{prop:fragmS4}, ${\{\vf\mid \vL[1]\vd \vf^{[*]} \}}$ is a logic containing $\LS{5}$.
Thus $x$ contains formulas $\Box^* \psi \imp   \Di^* \Box^* \psi$ and  $\Di^* \Box^* \vf\imp \Box^* \vf$.
Since $\vL\subseteq \vL[1]$,  $x$ also contains  $\Di^*\Box^* \psi\imp \Di^*\Box^* \vf$.
It follows that if $x$ contains $\Box^* \psi$, then $x$ contains $\Box^* \vf$.
By the Canonical model theorem,
$\vL[1]\vd \Box^* \psi\imp \Box^* \vf$.

Now suppose
$\vL[1] \vd \Box^* \psi\imp \Box^* \vf$.  Assume that  $\Di^*\Box^* \psi\in x$ for some element $x$ of $\frF$.
Then for some $y$ we have $\Box^* \psi \in y$ and $x R_\frF^* y$.
The set $R^*_\frF(y)$ has a maximal element $z$ by Proposition \ref{prop:max}. It follows that $\h(\cone{\frF}{z})=1$.
By Proposition \ref{prop:canon-pointwise}, $\frG=\frametop{\frF}{1}$. Thus $z$ is in $\frG$ and hence $\Box^*\psi\imp  \Box^* \vf$ is in $z$.
  Since $yR^*_\frF z$, we have $\Box^* \psi \in z$, which implies that  $\Box^* \vf \in z$.
Hence $\Di^*\Box^* \vf$ is in $x$.
It follows that  $\vL  \vd \Di^*\Box^* \psi\imp \Di^*\Box^* \vf$.
\end{proof}

\begin{theorem}\label{thm:embedd-corr}
Let $\vL$ be a pretransitive logic.
\begin{enumerate}
\item
For all $\vf$, we have
$\vL[1]\vd \vf$  iff $\vL \vd \Di^*\Box^*\vf.$
\item
If $\vL$ is decidable, then so is $\vL[1]$.
\item
If $\vL$ has the finite model property, then so does $\vL[1]$.
\end{enumerate}
\end{theorem}
\begin{proof}
By Theorem \ref{thm:embedd}, we have
$$
\vL[1] \vd \Box^* \top\imp \Box^* \vf \tiff \vL  \vd \Di^*\Box^*\top\imp \Di^*\Box^* \vf.$$
By Proposition \ref{prop:fragmS4}, we have $\top \lra \Box^* \top$ and $\top \lra  \Di^*\Box^*\top$ in every pretransitive logic;
also, we have {$\Box^* \vf \in\vL[1]$} iff $\vf \in \vL[1]$.
Now the first statement follows.

The second statement is an immediate consequence of the first one.

Suppose $\vL$ has the finite model property. Consider a formula $\vf\notin \vL[1]$. Then $\Di^*\Box^*\vf\notin \vL$. Then $\Di^*\Box^*\vf$ is refuted in some finite $\vL$-frame $\frF$.
If follows that $\vf$ is refuted in $\frF$ at some point in a maximal cluster $C$. The restriction $\frF\restr C$ is a generated subframe of $\frF$.
Thus $\frF\restr C$ refutes $\vf$ and validates $\vL$.
The height of this restriction is $1$, so $\frF\restr C\mo \vL[1]$. Thus $\vL[1]$  has the finite model property.
\end{proof}

\begin{example}
Important examples of pretransitive frames are birelational frames $(W,\leq,R)$ with transitive $R$.
Recall that $(W,\leq,R)$ is a {\em birelational frame}, if $\leq$ is a partial order on $W$, $R\subseteq W^2$, and
$$(R\circ{\leq})\;\subseteq \; ({\leq}\circ R), \quad (R^{-1}\circ{\leq})\; \subseteq \; ({\leq}\circ R^{-1}).$$
Consider the class of all birelational frames $(W,\leq,R)$ with transitive reflexive $R$.
Its modal  logic $\vL$ is the smallest bimodal logic containing the axioms of $\LS{4}$ for modalities $\Box_0,\Box_1$, and the formulas
$\Di_1\Di_0 p\imp \Di_0\Di_1 p$ and $\Di_0\Box_1 p\imp \Box_1\Di_0 p$
\cite{zakharyaschev_many-dimensional_2003}
(recall that in the semantics of modal intuitionistic logic, the logic of this class is known to be $\logicts{IS4}$ \cite{FS}, one of the
\emph{``most prominent logics for which decidability is left open''}
 \cite{SimpsIntModal94}).
In this case, $\Di_0\Di_1$ plays the role of the master modality, and the formula $B_1$ says that ${\leq}\circ R$ is an equivalence.
The decidability and the finite model property of the logic $\vL$, as well as of the logic $\logicts{IS4}$,  is an open question.
By the above theorem, we have
$\vL \vd \Di_0\Di_1\Box_0\Box_1 \vf$ iff $\vL[1]\vd \vf$.
\begin{question*}
Is the logic $\vL[1]$ decidable? Does it have the fmp?
\end{question*}
\end{example}

\section{Translation for logics of arbitrary finite height}\label{sec:main}

In the proof of Theorem \ref{thm:embedd} we used the following property of a canonical frame $\frF$ of $\vL$:
every point in  $\frF$  is  below  (w.r.t. to the preorder $R^*_\frF$) a maximal point;
maximal points form $\frametop{\frF}{1}$, the canonical frame of $\vL[1]$.
To describe translations from $\vL[h]$ to $\vL$ for $h>1$, we shall use the following analog of this property.
\begin{definition}
Let $0<h<\omega$. A frame $\frF$ is said to be {\em $h$-heavy} if for its every element $x$ which is not in $\frFh$
there exists $y$ such that $xR^*_\frF y$ and
$\h(\cone{\frF}{y})=h$.

$\frF$ is said to be {\em top-heavy} if it is $h$-heavy for all positive finite $h$.
\end{definition}

\begin{proposition}\label{prop:1heavy}
The $k$-canonical frame of a consistent pretransitive logic is $1$-heavy for every $k\leq \omega$.
\end{proposition}
\begin{proof}
In the Maximality lemma (Proposition \ref{prop:max}), put $\vf=\top$.
\end{proof}

It is known that $k$-canonical frames of unimodal transitive logics are
top-heavy for all finite $k$ (\cite{ShehtRigNish}, \cite{Fine85}, \cite{Bellissima1985}).\footnote{The
 term `top-heavy' was introduced in \cite{Fine85}.}
This can be generalized for the pretransitive case as follows.
\begin{theorem}\label{thm:topheavy}
Let  $\vL$ be a consistent pretransitive logic, $h,k< \omega$. If $\vL[h]$ is $k$-tabular, then:
\begin{enumerate}
\item
For every $i\leq h$,  there exists a formula $\Bik$ such that
$\Bik\in x$ iff   the depth of $x$ in the $k$-canonical frame of $\vL$  is less than or equal to $i$.
\item
The $k$-canonical frame of $\vL$ is $(h+1)$-heavy.
\end{enumerate}
\end{theorem}
\begin{proof}
The case $h=0$ follows from Proposition \ref{prop:1heavy}. Suppose $h>0$.

Let $\frF=(W,(R_i)_{i<n})$ be the $k$-canonical frame of $\vL$.
By Proposition \ref{prop:canon-pointwise},  the frame  $\frFh=(\Wh,(\Rh_i)_{i<n})$ is the $k$-canonical frame of $\vL[h]$.
Since $\vL[h]$ is $k$-tabular, it follows that $\Wh$ is finite and for every $a$ in $\Wh$ there exists a $k$-formula $\alpha(a)$
such that 
\begin{equation}\label{eq:atom}
\AA b \in \Wh\, ( \alpha(a)\in b \;\iff \;b=a).
\end{equation}
Without loss of generality we may assume that $\alpha(a)$ is of the form

\begin{equation}\label{eq:vars-a}
p_0^\pm\wedge\ldots\wedge p_{k-1}^\pm \wedge \vf,
\end{equation}
where $p_i^\pm\in\{p_i,\neg p_i\}$.

For $a\in \Wh$ let $\beta(a)$ be the following
Jankov-Fine formula:
\begin{equation}\label{eq:vars-beta}
\beta(a)=\alpha(a)\wedge \gamma,
\end{equation}
where $\gamma$ is the conjunction of the formulas
\begin{eqnarray}
\Box^*  \bigwedge \left\{ \alpha(b_1)\imp \Di_i \alpha(b_2) \mid  (b_1,b_2)\in \Rh_i,\; i<n    \right\}   \label{eq:Jank1}\\
\Box^*  \bigwedge \left\{ \alpha(b_1)\imp \neg\Di_i \alpha(b_2) \mid  (b_1,b_2)\in \Wh^2\setminus\Rh_i,\; \;i<n    \right\}\label{eq:Jank2}\\
\Box^* \vee \left\{ \alpha(b) \mid b\in \Wh \right\}\label{eq:Jank4}
\end{eqnarray}

For all $x,y\in W$, $i<n$ we have
\begin{equation}\label{eq:gamma-up}
\textrm{if } \gamma\in x \textrm{ and } xR_iy,  \textrm{ then } \gamma\in y.
\end{equation}

We claim that 
\begin{equation}\label{eq:definable}
\AA a\in \Wh\;\AA x\in W\; ( \beta(a)\in x \;\iff \;x=a).
\end{equation}
To prove this, by induction on the length of formulas
we show that for all $k$-formulas $\vf$, all $a\in \Wh$, and all $x\in W$,
\begin{equation}\label{eq:same-formulas}
\textrm{if } \beta(a)\in x, \textrm{ then } \vf\in a \iff \vf\in x.
\end{equation}
The basis of induction follows from (\ref{eq:vars-a}). The Boolean cases are trivial.

Assume that $\vf=\Di_i\psi$.

First, suppose $\Di_i\psi\in a$. We have $\psi\in b$ for some $b$ with $a\Rh_ib$.
Since $\beta(a) \in x$, by (\ref{eq:Jank1}) we have $\Di_i\alpha(b)\in x$.
Then we have $\alpha(b)\in y$ for some $y$ with $xR_iy$; by (\ref{eq:gamma-up}), $\beta(b)\in y$. Hence $\psi\in y$  by induction
hypothesis. Thus $\Di_i\psi\in x$.

Now let us show that $\Di_i\psi\in a$ whenever $\Di_i\psi\in x$. In this case we have $\psi\in y$ for some $y$ with $xR_iy$.
By (\ref{eq:Jank4}) we infer that $\alpha(b)\in y$ for some $b\in \Wh$.
Thus $\Di_i \alpha (b)\in x$. Since $\alpha(a)\in x$, it follows from (\ref{eq:Jank2}) that $a\Rh_ib$.
By (\ref{eq:gamma-up}) we have $\gamma\in y$, thus
$\beta(b)\in y$;
by induction
hypothesis $\psi\in b$. Hence $\Di_i \psi\in a$, as required.

Thus (\ref{eq:same-formulas}) is proved and (\ref{eq:definable})  follows.



Now using the formulas (\ref{eq:vars-beta}), for $i\leq h$ we can define the formulas  $\Bik$ such that
for all $x$ in $\frF$,
\begin{equation}
\text{the depth of $x$ in $\frF$ is $\leq i$}\;\tiff\;\Bik\in x.
\end{equation}
For this, we put
\begin{equation}
\Bik=\bigvee\{\beta(a)\mid \h(\cone{\frF}{a})\leq i\}.
\end{equation}
This proves the first statement of the theorem.

In particular, it follows that $W\setminus \Wh$ is definable in the $k$-canonical model of $\vL$:
$$
x\in W\setminus \Wh \tiff  \neg \B_{h,k}\in x.
$$
Now by Proposition \ref{prop:max} we have that if $x$ is not in $\Wh$, then
there exists a maximal $y$ in $R^*_\frF(x)\setminus \Wh$.
Hence if $(y,z)\in R_\frF^*$ and $(z,y)\notin R_\frF^*$ for some $z$, then $z$ belongs to $\Wh$, which means $\h(\cone{\frF}{z})\leq h$.
Thus $\h(\cone{\frF}{y})\leq h+1$. On the other hand, $y\notin \Wh$.
It follows that $\h(\cone{\frF}{y})=h+1$, as required.
\end{proof}

The logic $\vL[0]$ is inconsistent, so it is $k$-tabular. Hence Proposition \ref{prop:1heavy} can be considered as a particular case of the above theorem.

Note that formulas (\ref{eq:vars-beta}) define atoms in the Lindenbaum-Tarski (i.e., free)  $k$-generated algebra of $\vL$.


\begin{theorem}\label{thm:main}
Let $\vL$ be a pretransitive logic, $h,k< \omega$. If $\vL[h]$ is $k$-tabular,  then for all $k$-formulas $\vf$ we have
\begin{equation}\label{eq:translMain}
\vL[h+1] \vd\vf \tiff \vL \vd \bigwedge_{i\leq h} (\Box^*(\Box^* \vf\imp \Bik)\imp \Bik).
\end{equation}
\end{theorem}
\begin{proof}
We may assume that both $\vL$ and $\vL[h+1]$ are consistent (Proposition \ref{prop:cons}).
Let $\frF$ be the $k$-canonical frame of $\vL$.

Suppose $\vL[h+1]\vd\vf$. We claim that for all $i \leq h$,
$\Box^*(\Box^* \vf\imp \Bik)\imp \Bik$
is true at every point $x$ in the $k$-canonical model of $\vL$.
Let $\neg \Bik$ be in $x$. Let us show that $\neg \Bik \wedge \Box^*\vf\in y$ for some $y$ with $xR_\frF^*y$.
First, assume that $x$ is in $\frametop{\frF}{h+1}$. By Proposition \ref{prop:canon-pointwise},  $x$ contains $\vL[h+1]$.
Since $\vf\in \vL[h+1]$, we have $\Box^*\vf \in \vL[h+1]$. Thus $\Box^*\vf\in x$. Since $R^*_\frF$ is reflexive, in this case
we can put $y=x$.
%
Second, suppose $x$ is not in $\frametop{\frF}{h+1}$. By Theorem \ref{thm:topheavy}, there exists $y$ such that $xR_\frF^*y$ and $\h(\cone{\frF}{y})=h+1$.
We have $\Box^*\vf\in y$ and $\Bik\notin y$.  This proves the ``only if'' part.

Now suppose that
$\vL \vd \Box^*(\Box^* \vf\imp \Bik)\imp \Bik$
for all $i\leq h$.
Assume $\h(\cone{\frF}{x})=i \leq h+1$. In this case $\B_{i-1,k} \notin x$. Since
\todo{DC}
$\Box^*(\Box^* \vf\imp \B_{i-1,k})\imp \B_{i-1,k}$ is in $x$,
it follows that  $\neg \B_{i-1,k} \wedge \Box^*\vf$ is in $y$ for some $y$ with $xR_\frF^* y$. The first conjunct says that $y$ is not in $\frametop{\frF}{i-1}$.
Since $y$ is in $\frametop{\frF}{i}$, it follows that  $\h(\cone{\frF}{y})=i$.
Hence $y$ and $x$ belong to the same cluster. Since $\Box^*\vf\in y$, we obtain $\vf\in x$.
It follows that $\vf\in x$ for all $x$ in $\frametop{\frF}{h+1}$. By Proposition \ref{prop:canon-pointwise}, $\vL[h+1]\vd\vf$.
\end{proof}

Note that $\B_{0,k}$ is $\bot$ for all $k<\omega$.
Thus, (\ref{eq:translMain}) generalizes the translation described in Theorem \ref{thm:embedd-corr}.

\smallskip

Theorem \ref{thm:main} provides translations for the case when $\vL$ is a unimodal transitive logic
(recall that transitive logics of finite height are locally tabular \cite{Seg_Essay}).
It should be noted that in this case
Theorem \ref{thm:topheavy}, the key ingredient of the proof of Theorem \ref{thm:main}, has been known since 1970s:
formulas $\Bik$ in transitive canonical frames were described in \cite{ShehtRigNish} (see also \cite{Fine85}, \cite{Bellissima1985}).

An analog of Theorem \ref{thm:main} can be formulated for intermediate logics.
Formulas $\Bik^\Int$ defining points of finite depth in finitely generated intuitionistic canonical frames were described
in \cite{ShehtPhD} (see also \cite{Bellissima1986Ht}).
Similarly to the proof of Theorem  \ref{thm:main},
 it can be shown that
$$\IL[h+1]\vd\vf  \tiff  \IL\vd \bigwedge_{i\leq h} ((\vf\imp \Bik^\Int)\imp \Bik^\Int)$$
for all finite $h$ and for all $k$-formulas $\vf$.

\section{Corollaries, examples, and open problems}\label{sec:coroll}

The translation (\ref{eq:translMain})  holds for all finite $h,k$ in the case when $\vL$ is a transitive unimodal logic.
Indeed, by the Segerberg -- Maksimova criterion (Theorem \ref{thm:seg-max-crit}), a transitive logic is locally tabular iff it is of finite height.
This criterion was recently generalized to a wide family of pretransitive logics \cite{LocalTab16AiML}. For example,
if a unimodal $\vL$ contains the formula $\Di^{m+1} p \imp \Di p\vee p$ for some $m>0$, then  $\vL$ is locally tabular iff it is of finite height.
Thus, (\ref{eq:translMain}) holds for all finite $h,k$ in this case too.

However, in general $k$-tabularity of $\vL[h]$ depends both on $h$ and on $k$.
\begin{example}
Consider the smallest reflexive 2-transitive logic $\LK{}+\{p\imp \Di p, \Di^3 p\imp \Di^2 p\}$ and
its extension $\vL$ with the McKinsey formula for the master modality, $\Box^2 \Di^2 p\imp \Di^2 \Box^2 p$.
Maximal clusters in the canonical frames of $\vL$ are reflexive singletons, so $\vL[1]=\LK{}+p\lra \Box p$
by Proposition \ref{prop:canon-pointwise}. Clearly, $ \vL[1]$ is locally tabular.
It follows that we have the translation (\ref{eq:translMain}) from $\vL[2]$ to $\vL$ for all finite $k$.

However,  $\vL[2]$ is not even $1$-tabular. To see this,
consider the frame $\frF_0=(\omega,R_0)$, where
$$x R_0 y \quad \tiff \quad x \neq y+1 \textrm{ and } y\neq x+1.$$
Let $\frF=(\omega+1,R)$, where $xRy$ iff $xR_0y$ or $y=\omega$. Clearly, $\frF\mo\vL[2]$.
Consider a model $\mM$ on $\frF$ such that $x\mo p_0$ iff $x=0$ or $x= \omega$.
Put $\alpha_0=p_0\wedge \Di\neg p_0$, $\alpha_1=\neg \Di \alpha_0\wedge \neg p_0$, and
$\alpha_{i+1}=\neg (\Di \alpha_i\vee\alpha_{i-1}) \wedge \neg p_0$ for $i>0$. By an easy induction,
in $\mM$ we have for all $i$: $x\mo \alpha_i$ iff $x=i$. Thus if $i\neq j$, then $\alpha_i\lra \alpha_j\notin \vL$.
\end{example}
It is not difficult to construct other examples of this kind
for arbitrary finite $h$: there are
pretransitive logics such that $\vL[h]$ is locally tabular,
and $\vL[h+1]$ is not one-tabular.

With the parameter $k$, the situation is much more intriguing.  The following result was proved in \cite{Max1975}:
\begin{equation}\label{eq:maxOneTab}
\text{ A unimodal transitive logic is locally tabular iff
it is 1-tabular. }
\end{equation}
The recent results \cite{LocalTab16AiML} show that this equivalence also holds for many non-transitive logics. For example, if a unimodal $\vL$ contains
$\Di^{m+1} p \imp \Di p\vee p$ for some $m>0$, then it is locally tabular iff it is 1-tabular.
The  question whether this equivalence holds for every modal logic
has been open since  1970s.
\begin{theorem}
There exists a unimodal 1-tabular logic $\vL$ which is not locally tabular.
\end{theorem}
\begin{proof}[Proof (sketch).]
Let $\vL$ be the logic of the frame
$(\omega+1,R)$, where
$$xRy \tiff  x\leq y  \textrm{ or }x=\omega.$$

First, we claim that $\vL$ is not locally tabular.

The following fact follows from Theorem 4.3 and Lemma 5.9 in
\cite{LocalTab16AiML}:
if the logic of a frame $\frF$ is locally tabular, then the logic of an arbitrary restriction $\frF\restr V$ of $\frF$ is locally tabular.

The restriction of $(\omega+1,R)$ onto $\omega$ is the frame $(\omega,\leq)$, which is not locally tabular (it is of infinite height).
Thus $\vL$ is not locally tabular.

To show that $\vL$ is 1-tabular, we need the following observation.
If every $k$-generated subalgebra of an algebra $\frA$ contains at most $m$ elements for some fixed $m<\omega$, then
the free $k$-generated algebra in the variety generated by $\frA$ is finite; see \cite{Malcev}.

Consider the complex algebra $\frA$ of the frame $(\omega+1,R)$. One can check that
every $1$-generated subalgebra of $\frA$ contains at most 8 elements. By the above observation,  $\vL$ is 1-tabular.
\end{proof}

It is unknown whether 2-tabularity of a modal logic implies its local tabularity.
At least, does $k$-tabularity  imply  local tabularity, for some fixed $k$ for  all unimodal logics?
The same questions are open in the intuitionistic case \cite[Problem 2.4]{GuramRevazProblem}.

\smallskip

Finite height is not a necessary condition for local tabularity of intermediate logics. What can be an analog of Gliveko's
translation in the case of a locally tabular intermediate logic with no finite height axioms?
Another generalization can  probably be found in the area of modal intuitionistic logics.
In \cite{GuramGlivenko}, Glivenko type theorems were proved  for extensions of the logic $\MIPC$;
in
\cite{Guram-Revaz}, local tabularity of these extensions was considered.
What can be an analog of Theorem \ref{thm:main} for modal intuitionistic logics?

\smallskip

In \cite{Rybakov1992}, Glivenko's theorem was used to obtain decidability (and the finite model property) for extensions of $\LS{4}$ with $\Box\Di$-formulas
(such formulas are built from literals $\Box\Di p_i$). An analog of this result
can be obtained for extensions of a pretransitive logic $\vL$ in the case when
$\vL$ is decidable  (or has the finite model property) and
$\vL[1]$ is locally tabular.

\section{Acknowledgements}
I would like to thank Valentin Shehtman for many useful discussions and comments.
I would also like to thank the three anonymous reviewers for their
suggestions and questions on the early version of the manuscript.

The work on this paper was supported by the Russian Science Foundation under grant 16-11-10252 and carried out at Steklov Mathematical Institute of Russian Academy of Sciences.

\bibliographystyle{eptcs}
\bibliography{gliv-bib}

\end{document}